\documentclass[10pt]{article}

\usepackage{lmodern}
\usepackage[T1]{fontenc}
\usepackage{amsmath}
\usepackage{amsthm}
\usepackage{amssymb}
\usepackage{mathabx} 
\usepackage{url}
\usepackage{latexsym}
\usepackage{titlefoot}
\usepackage[small]{titlesec}
\usepackage{units} 
\usepackage[small,it]{caption}
\usepackage{xspace}

\usepackage[square,comma,numbers,sort&compress]{natbib}

\usepackage{lineno}

\usepackage{graphicx} 

\usepackage{tikz}

\usepackage{color}
\definecolor{lightgray}{rgb}{0.8, 0.8, 0.8}
\definecolor{darkgray}{rgb}{0.6, 0.6, 0.6}

\usepackage[bookmarks]{hyperref}
\hypersetup{
	colorlinks=true,
	linkcolor=black,
	anchorcolor=black,
	citecolor=black,
	urlcolor=black,
	pdfpagemode=UseThumbs,
	pdftitle={On the Growth of Grid Classes of Permutations},
	pdfsubject={Combinatorics},
	pdfauthor={Michael Albert and Vincent Vatter},
}

\newcounter{todocounter}

\theoremstyle{plain}
\newtheorem{theorem}{Theorem}
\newtheorem{proposition}[theorem]{Proposition}

\makeatletter
\newfont{\footsc}{cmcsc10 at 8truept}
\newfont{\footbf}{cmbx10 at 8truept}
\newfont{\footrm}{cmr10 at 10truept}
\renewenvironment{abstract}{
	\begin{list}{}%
	{\setlength{\rightmargin}{1in}%
	\setlength{\leftmargin}{1in}}%
	\item[]\ignorespaces\begin{small}}%
	{\end{small}\unskip\end{list}%
}

\newcommand{\Av}{\operatorname{Av}}
\newcommand{\Grid}{\operatorname{Grid}}

\newcommand{\C}{\mathcal{C}}

\newcommand{\M}{\mathcal{M}}
\newcommand{\gr}{\mathrm{gr}}

\newcommand{\st}{\::\:}
\newcommand{\gridded}{\sharp}
\newcommand{\supp}{\mathrm{supp}}
\renewcommand{\vec}[1]{\mathbf{#1}}
\newcommand{\fnmatrix}[2]{\text{\begin{footnotesize}$\left(\begin{array}{#1}#2\end{array}\right)$\end{footnotesize}}}

\newcommand\dummyvar{\mathpalette\bigcdot@{.75}}
\newcommand\bigcdot@[2]{\mathbin{\vcenter{\hbox{\scalebox{#2}{$\m@th#1\bullet$}}}}}

\newcommand\mybullet{\raisebox{-5pt}{\normalsize \ensuremath{\bullet}}}
\newcommand\mycirc{\raisebox{-5pt}{\normalsize \ensuremath{\circ}}}

\makeatletter
\def\absdot{\@ifnextchar[{\@absdotlabel}{\@absdotnolabel}}
	\def\@absdotlabel[#1]#2{%
		\node at #2 {\normalsize \mybullet};
		\node at #2 [below=2pt] {\ensuremath{#1}};
	}
	\def\@absdotnolabel#1{%
		\node at #1 {\normalsize \mybullet};
	}
\def\absdothollow{\@ifnextchar[{\@absdothollowlabel}{\@absdothollownolabel}}
	\def\@absdothollowlabel[#1]#2{%
		\node at #2 {\normalsize \textcolor{white}{\mybullet}};
		\node at #2 {\normalsize \mycirc};
		\node at #2 [below=2pt] {\ensuremath{#1}};
	}
	\def\@absdothollownolabel#1{%
		\node at #1 {\normalsize \textcolor{white}{\mybullet}};
		\node at #1 {\normalsize \mycirc};
	}
\makeatother

\newcommand{\plotperm}[1]{
	\foreach \j [count=\i] in {#1} {
		\absdot{(\i,\j)};
	};
}

\newcommand{\plotpermbox}[4]{
	\draw [darkgray, thick, line cap=round]
		({#1-0.5}, {#2-0.5}) rectangle ({#3+0.5}, {#4+0.5});
}

\newcommand{\plotpermgraph}[1]{
	\foreach \j [count=\i] in {#1} {
		\foreach \b [count=\a] in {#1} {
			\ifthenelse{\a<\i \AND \b>\j}{\draw (\a,\b)--(\i,\j);}{}
		};
	};
	\plotperm{#1};
}

\newcommand{\plotpermdyckpath}[1]{
	\draw[ultra thick, line cap=round] (0.5,0.5)
	\foreach \step in {#1} {
		\ifnum\step=1
			-- ++(0,1)
		\else
			-- ++(1,0)
		\fi
	};
}

\newcommand{\plotdyckpath}[1]{
	\draw[ultra thick, line cap=round] (0.5,0)
	\foreach \step in {#1} {
		\ifnum\step=1
			-- ++(1,1)
		\else
			-- ++(1,-1)
		\fi
	};
}

\newcommand{\arcskinnyplain}[2]{
	\draw[thick] (#1,0) arc (180:0:{(#2-#1)/2});
}

\newcommand{\matchsmall}[1]{
	\begin{tikzpicture}[scale=.1, anchor=base]
		\def\h{0};
		\def\maxh{0};
		\foreach \i/\j in {#1} {
			\pgfmathparse{\j-\i};
			\let\h\pgfmathresult;
			\pgfmathifthenelse{\h>\maxh}{\h}{\maxh};
			\global\let\maxh\pgfmathresult;
		};
		\pgftransformyscale{{4.5/\maxh}};
		\foreach \i/\j in {#1} {
			\arcskinnyplain{\i}{\j};
		};
	\end{tikzpicture}
}

\newcommand{\matchpermsmall}[1]{
	\begin{tikzpicture}[scale=.1, anchor=base]
		\foreach \j [count=\n] in {#1} {};
		\def\h{0};
		\def\maxh{0};
		\foreach \j [count=\i] in {#1} {
			\pgfmathparse{2*\n+1-\j-\i};
			\let\h\pgfmathresult;
			\pgfmathifthenelse{\h>\maxh}{\h}{\maxh};
			\global\let\maxh\pgfmathresult;
		};
		\pgftransformyscale{{4.5/\maxh}};
		\foreach \j [count=\i] in {#1} {
			\arcskinnyplain{\i}{{2*\n+1-\j}};
		};
	\end{tikzpicture}
}

\newcommand{\addomain}{\vec{A}}

\datefoot{\today}
\amssubj{05A05, 05A15}

\newpagestyle{main}[\small]{
	\headrule
	\sethead[\usepage][][]
	{\sc On the Growth of Grid Classes of Permutations}{}{\usepage}
}

\title{\sc An Elementary Proof of Bevan's Theorem On the Growth of Grid Classes of Permutations}

\author{\centering
\begin{tabular}{ccc}
Michael Albert
&\rule{0pt}{0pt}&
Vincent Vatter%
\footnote{Vatter's research was partially supported by the National Science Foundation under Grant Number DMS-1301692 and the National Security Agency under Grant Number H98230-16-1-0324. The United States Government is authorized to reproduce and distribute reprints not-withstanding any copyright notation herein.}\\[-0.25ex]
\small Department of Computer Science
&&
\small Department of Mathematics\\[-0.5ex]
\small University of Otago
&&
\small University of Florida\\[-0.5ex]
\small Dunedin, New Zealand
&&
\small Gainesville, Florida USA\\[-1.5ex]
\end{tabular}
}

\titleformat{\section}{\large\sc}{\thesection.}{1em}{}
\date{}
 
\begin{document}
\maketitle

\pagestyle{main}

\begin{abstract}
{\sc Abstract.}
Bevan established that the growth rate of a monotone grid class of permutations is equal to the square of the spectral radius of a related bipartite graph. We give an elementary and self-contained proof of a generalization of this result using only Stirling's Formula, the method of Lagrange multipliers, and the singular value decomposition of matrices.
\end{abstract}

\section{Introduction}

Herein we provide an elementary derivation of the asymptotic enumeration of certain permutation classes called grid classes. For a broad overview of permutation classes, we refer the reader to the second author's survey~\cite{vatter:permutation-cla:}, and give only the necessary definitions here.

Two finite sequences of distinct real numbers are \emph{order isomorphic} if they have the same relative comparisons. In other words, the sequences $a(1),\dots,a(k)$ and $b(1),\dots,b(k)$ are order isomorphic if,  for all $i$ and $j$, $a(i)<a(j)$ if and only if $b(i)<b(j)$. A permutation \emph{of length $n$} is a bijective map from the set $\{1,2,\dots,n\}$ to itself. The permutation $\pi$ of length $n$, thought of in one-line notation as the sequence $\pi(1)\cdots\pi(n)$, is said to \emph{contain} the permutation $\sigma$ if $\pi$ has a subsequence which is order isomorphic to $\sigma$. Otherwise we say that $\pi$ \emph{avoids} $\sigma$. For example, the permutation $41523$ contains $231$ because of its subsequence $452$ but avoids $321$ because it does not contain three entries in decreasing order.

This containment relation is a partial order on the set of all finite permutations, and a \emph{permutation class} is a downset in this order. That is, if $\C$ is a permutation class, $\pi\in\C$, and $\sigma$ is contained in $\pi$, then $\sigma\in\C$. Permutation classes are frequently described via the minimal permutations they do \emph{not} contain (or, in other words, \emph{avoid}), for which we use the notation
\[
	\Av(B)=\{\pi\st\pi\mbox{ avoids all $\beta\in B$}\}.
\]
We denote by $\C_n$ the set of permutations of length $n$ in the class $\C$ and define the \emph{proper growth rate} of $\C$ by
\[
	\gr(\C)=\lim_{n\rightarrow\infty} \sqrt[n]{|\C_n|}
\]
if this limit exists (it is conjectured that this limit always exists; in cases where it is not known to exist we often instead consider the \emph{upper growth rate} of $\C$, obtained by replacing the limit above by a limit superior). More generally, for any collection of discrete structures we define the growth rate (or upper growth rate) to be the limit (or limit superior) of the $n^{\mbox{\scriptsize th}}$ root of the number of structures of size $n$ in the collection as $n$ tends to infinity. From the resolution of the Stanley--Wilf Conjecture by Marcus and Tardos~\cite{marcus:excluded-permut:} it follows that---except for the degenerate case of the class of all permutations---all upper growth rates of permutation classes are finite.

We are interested in a particular construction of permutation classes, called \emph{(generalized) grid classes}, which were first introduced in the second author's work on the classification of the set of growth rates of permutation classes~\cite{vatter:small-permutati:}. 

\begin{figure}
\begin{footnotesize}
\begin{center}
	\begin{tikzpicture}[scale=0.2, baseline=(current bounding box.center)]
		\plotpermbox{1}{1}{7}{10};
		\plotpermbox{1}{11}{7}{16};
		\plotpermbox{8}{1}{16}{10};
		\plotpermbox{8}{11}{16}{16};
		\plotperm{14,1,5,7,12,10,11,9,13,15,8,6,4,16,3,2};
	\end{tikzpicture}
\quad\quad
	\begin{tikzpicture}[scale=0.2, baseline=(current bounding box.center)]
		\plotpermbox{1}{1}{13}{7};
		\plotpermbox{1}{8}{13}{16};
		\plotpermbox{14}{1}{16}{7};
		\plotpermbox{14}{8}{16}{16};
		\plotperm{10,12,8,6,5,13,4,15,9,2,10,16,14,7,3,1};
	\end{tikzpicture}
\end{center}
\end{footnotesize}
\caption{An $\protect\fnmatrix{cc}{\Av(12)&\Av(21)\\\Av(21)&\Av(12)}$-gridding (left) and an $\protect\fnmatrix{cc}{\Av(321)&\emptyset\\\Av(12)&\Av(12)}$-gridding (right) of two different permutations.}
\label{fig-gridding-examples}
\end{figure}

Suppose that $\M$ is a $t\times u$ matrix whose entries are permutation classes.  An {\it $\M$-gridding\/} of the permutation $\pi$ of length $n$ is a pair of sequences $1=c_1\le\cdots\le c_{t+1}=n+1$ (the column divisions) and $1=r_1\le\cdots\le r_{u+1}=n+1$ (the row divisions) such that for all $1\le k\le t$ and $1\le\ell\le u$, the entries of $\pi$ with indices in $[c_k,c_{k+1})$ and values in $[r_{\ell}, r_{\ell+1})$ are order isomorphic to a permutation in the class $\M_{k,\ell}$. The \emph{grid class of $\M$}, denoted $\Grid(\M)$, is defined as the class of all permutations which possess an $\M$-gridding.

The \emph{skew-merged} permutations were the first grid class of permutations to be studied. Stankova~\cite{stankova:forbidden-subse:} originally defined these to be the permutations whose entries can be partitioned into an increasing subsequence and a decreasing subsequence. This condition is equivalent to the existence of a horizontal and a vertical line which slice the plot of $\pi$ (the set $\{(i,\pi(i))\}$ of points in the plane) into four quadrants such that the northwestern and southeastern cells consist of decreasing subsequences (i.e., subsequences avoiding $12$) and the southwestern and northeastern cells consist of increasing subsequences (which therefore avoid $21$). This shows that the class of skew-merged permutations can be expressed as the grid class
\[
	\Grid\fnmatrix{cc}{\Av(12)&\Av(21)\\\Av(21)&\Av(12)}.
\]
An example of such a gridding of a skew-merged permutation is shown on the left-hand side of Figure~\ref{fig-gridding-examples}. The right-hand side of that figure shows a gridding of a member of a different grid class. While it makes no difference to our results, we use Cartesian coordinates to index matrices, so the entry $\M_{k,\ell}$ denotes the cell in the $k^{\mbox{\scriptsize th}}$ column from the left and the $\ell^{\mbox{\scriptsize th}}$ row from the bottom of $\M$.

The main theorem of this paper, below, shows that the proper growth rate of $\Grid(\M)$ is an eigenvalue of a closely related matrix.

\begin{theorem}
\label{thm-bevan-grid-gr}
Let $\M$ be a $t\times u$ matrix of permutation classes, each with a proper growth rate, and let $\Gamma$ be the $t\times u$ matrix with $\Gamma_{k,\ell}=\sqrt{\gr(\M_{k,\ell})}$. The proper growth rate of $\Grid(\M)$ exists and is equal to the greatest eigenvalue of $\Gamma^T\Gamma$ \textup{(}or equivalently, of $\Gamma\Gamma^T$\textup{)}.
\end{theorem}

Before beginning the proof, a few comments are in order. First, $\Gamma^T\Gamma$ is obviously a symmetric positive semidefinite matix, so the eigenvalue occurring in the statement of Theorem~\ref{thm-bevan-grid-gr} is guaranteed to be nonnegative. Second, if the cells of $\M$ do not have proper growth rates, one can take $\Gamma_{k,\ell}$ equal to the square root of the upper growth rate of $\M_{k,\ell}$ and, mutatis mutandis, the argument used to prove Theorem~\ref{thm-bevan-grid-gr} gives an upper bound on the upper growth rate of $\Grid(\M)$. However, one has no guarantee that the quantity so obtained is in fact the upper growth rate of $\Grid(\M)$.

More significantly, Theorem~\ref{thm-bevan-grid-gr} is a generalization of the work of Bevan. In \cite{bevan:growth-rates-of:}, he showed how to compute the proper growth rates of \emph{monotone} grid classes, which are grid classes in which the cells may contain only the class of increasing permutations (those avoiding $21$), the class of decreasing permutations (those avoiding $12$), or the empty class. Thus in applying Theorem~\ref{thm-bevan-grid-gr} to such classes, $\Gamma$ is a $0/1$ matrix. Via a delicate argument, Bevan showed that in this case, the growth rate of $\Grid(\M)$ is equal to the square of the spectral radius of a certain bipartite graph associated to $\Gamma$. Translated to a purely linear algebraic context, this result says that the growth rate of $\Grid(\M)$ is the square of the greatest eigenvalue of the matrix
\[
	\fnmatrix{cc}{0&\Gamma\\\Gamma^T&0},
\]
and it is a routine exercise to show that this formula is equivalent to Theorem~\ref{thm-bevan-grid-gr} in these cases.

Finally, because of this connection to spectral radii of graphs, the extensive literature on algebraic graph theory can be applied directly to the characterization of growth rates of monotone grid classes. We refer the reader to Bevan~\cite[Section 4]{bevan:growth-rates-of:} for these implications. Further applications of Theorem~\ref{thm-bevan-grid-gr}, to growth rates of merges of permutation classes, are described in Albert, Pantone, and Vatter~\cite{albert:on-the-growth-o:}. We note that special cases of some aspects of our approach are stated without proof in Bevan's thesis~\cite[Chapter 6]{bevan:on-the-growth-o:}.

\section{Initial Considerations}

A permutation $\pi\in\Grid(\M)$ together with a fixed $\M$-gridding of it is called an \emph{$\M$-gridded permutation}, and we denote by $\Grid^\gridded(\M)$ the collection of all $\M$-gridded permutations. Every permutation $\pi \in \Grid(\M)$ has at least one gridding that witnesses its membership in the grid class, so trivially $|\Grid_n(\M)|\le |\Grid^\gridded_n(\M)|$. Conversely, there are at most $n+1$ positions in which a vertical (resp., horizontal) line can be placed in a permutation of length $n$, and thus we have
\[
	\lvert \Grid_n(\M) \rvert
	\le
	\lvert \Grid^{\gridded}_n(\M) \rvert
	\le
	(n+1)^{t+u} \lvert \Grid_n(\M) \rvert.
\]
In particular, because $(n+1)^{(t+u)/n} \to 1$ as $n \to \infty$, the upper growth rates of $\Grid(\M)$ and $\Grid^{\gridded}(\M)$ are the same. For this reason, we work only in the gridded context for the rest of the proof.

Because we think of $\M$ as providing a recipe for constructing the members of $\Grid(\M)$, it is natural to refer to the individual positions in the matrix $\M$ as its \emph{cells}. More precisely (though we will not use this precise formulation in what follows) we can take the cells of $\M$ to consist simply of the pairs $(k,\ell)$ with $1 \leq k \leq t$ and $1 \leq \ell \leq u$, and we say that the cell $(k,\ell)$ contains the class $\M_{k,\ell}$. In fact we  implicitly identify the cell with its corresponding class.

By an argument similar to the preceding one, we see that cells of $\M$ which contain finite classes do not affect the growth rate of $\Grid(\M)$. Thus, we may assume that every cell of $\M$ is either empty or contains an infinite permutation class. Since any infinite class contains at least one permutation of every length, the growth rate of such classes is at least $1$.

\section{Translation to a Continuous Problem}


For the rest of the note we assume $\M$ is a fixed $t\times u$ matrix consisting of empty and infinite permutation classes, each with a proper growth rate. Given a real matrix $A$, its \emph{support} is the set of indices of cells corresponding to nonzero entries,
\[
	\supp(A)=\{(k,\ell)\st A_{k,\ell}\neq 0\}.
\]
Thus under our hypothesis that the entries of $\M$ are all empty or infinite classes, we see that $\M_{k,\ell}\neq\emptyset$ if and only if $(k,\ell)\in\supp(\Gamma)$. We say that a $t\times u$ real matrix $A$ is \emph{admissible} (implicitly, for $\M$) if
\[
	\supp(A)\subseteq\supp(\Gamma),
\]
or equivalently, if $A_{k,\ell}=0$ whenever $\M_{k,\ell}=\emptyset$. We further refer to the sum of the entries of a real matrix as its \emph{weight}.

If $A$ is an admissible nonnegative integral matrix of weight $n$, we denote by $\Grid^{\gridded}_A(\M)$ the subset of $\Grid^{\gridded}_n(\M)$ consisting of those gridded permutations such that the number of entries belonging to cell $(k,\ell)$ is $A_{k,\ell}$. The sets $\Grid^{\gridded}_A(\M)$ allow us to express the collection of $\M$-gridded permutations as a disjoint union
\[
	\Grid^\gridded_n(\M)=\biguplus \Grid^\gridded_A(\M),
\]
taken over all admissible nonnegative integral matrices of weight $n$.

For the rest of the note we adopt several notational conventions. First, we use the combinatorial interpretation $0^0=1$ when this quantity arises. Second, recalling that we use Cartesian indexing, for any matrix $A$ with real entries we denote by $\sum A_{k,\dummyvar}$ (resp., $\sum A_{\dummyvar,\ell}$) the sum of the entries in its $k^{\mbox{\scriptsize th}}$ column (resp., $\ell^{\mbox{\scriptsize th}}$ row). Similarly, we denote by $\prod A_{k,\dummyvar}$ (resp., $\prod A_{\dummyvar,\ell}$) the product of the nonzero entries in the $k^{\mbox{\scriptsize th}}$ column (resp., $\ell^{\mbox{\scriptsize th}}$ row) of $A$.

\begin{proposition}
\label{prop-count-gridded-A}
For every admissible nonnegative integral matrix $A$ we have
\[
	\lvert\Grid^\gridded_A(\M)\rvert
	=
	\prod_{k=1}^t \binom{\sum A_{k,\dummyvar}}{A_{k,1}, \dotsc, A_{k,u}} 
	\prod_{\ell=1}^u \binom{\sum A_{\dummyvar,\ell}}{A_{1,\ell}, \dotsc, A_{t,\ell}}
	\prod_{k,\ell} \lvert (\M_{k,\ell})_{A_{k,\ell}} \rvert.
\]
\end{proposition}

\begin{proof}
The multinomial coefficients in the first product count the number of ways in which the entries of the gridded permutation can be distributed horizontally within each column. Those in the second product enumerate the analogous possibilities for the rows. This determines the horizontal and vertical positions which are occupied in each cell and it remains only to choose the permutations formed by the contents of each of these cells, which may be order isomorphic to any permutation of length $A_{k,\ell}$ in the class $\M_{k,\ell}$.
\end{proof}

For each positive integer $n$, let $A^*_n$ denote an admissible nonnegative integral matrix of weight $n$ which maximizes $|\Grid^\gridded_A(\M)|$ over all admissible nonnegative integral matrix of weight $n$. Since $\Grid^\gridded_n(\M)$ is the disjoint union of $\Grid^\gridded_A(\M)$ over all such matrices, and because there are at most $(n+1)^{tu}$ such matrices of weight $n$, we have the bound
\[
	\lvert\Grid^\gridded_{A^*_n}(\M)\rvert
	\le
	\lvert\Grid^\gridded_n(\M)\rvert
	\le
	(n+1)^{tu}\,\lvert\Grid^\gridded_{A^*_n}(\M)\rvert.
\]
The polynomial factor $(n+1)^{tu}$ has no effect on upper growth rates, and thus we deduce the following.

\begin{proposition}
There is a sequence $\{A^*_n\}$ of admissible nonnegative integral matrices, each of weight $n$, such that the upper growth rate of $\Grid(\M)$ is equal to
\[
	\limsup_{n\rightarrow\infty}\ \lvert\Grid^\gridded_{A^*_n}(\M)\rvert^{1/n}.
\]
\end{proposition}

Given an admissible nonnegative integral matrix, we scale it by dividing each of its entries by its weight. Each scaled matrix obtained by this process lies in the \emph{admissible domain} $\addomain\subseteq [0,1]^{t\times u}$ consisting of all admissible matrices $\hat{A}$ of unit weight. On the other hand, given an arbitrary $\hat{A}\in\addomain$, the matrix $\lfloor n\hat{A}\rfloor$ defined by $\lfloor n\hat{A}\rfloor_{k,\ell}=\lfloor n\hat{A}_{k,\ell}\rfloor$ is an admissible nonnegative integral matrix of weight at most $n$.

Next we define a function $f:\addomain\to\mathbb{R}$ which is related to an approximation of the formula for $|\Grid^\gridded_A(\M)|$ given by Proposition~\ref{prop-count-gridded-A}. To motivate its definition, consider that for a matrix $A$ whose nonzero entries are large, Stirling's formula shows that
\[
	\binom{\sum A_{k,\dummyvar}}{A_{k,1}, \dotsc, A_{k,u}}
	\sim
	\sqrt{\frac{\sum A_{k,\dummyvar}}{(2\pi)^{u-1}\prod A_{k,\dummyvar}}}
	\cdot
	\frac{\left(\sum A_{k,\dummyvar}\right)^{(\sum A_{k,\dummyvar})}}{A_{k,1}^{A_{k,1}}\cdots A_{k,u}^{A_{k,u}}}.
\]
As we will eventually take $n^{\mbox{\scriptsize th}}$ roots, the quantity in the square root above will not contribute to the growth rate. By ignoring this soon-to-be-negligible quantity, we are left with
\[
	\frac{\left(\sum A_{k,\dummyvar}\right)^{(\sum A_{k,\dummyvar})}}{A_{k,1}^{A_{k,1}}\cdots A_{k,u}^{A_{k,u}}}.
\]
In forming the product in Proposition~\ref{prop-count-gridded-A}, each row contributes the numerator of one of these fractions, each column contributes a similar numerator, and the term $A_{k,\ell}^{A_{k,\ell}}$ occurs in precisely two of the denominators. Each nonzero cell $(k,\ell)$ also contributes the factor $\left| (\M_{k,\ell})_{A_{k,\ell}} \right|$, which converges after taking $n^{\mbox{\scriptsize th}}$ roots to $\Gamma_{k,\ell}^{2A_{k,\ell}}$ because the growth rate of $\M_{k,\ell}$ is $\Gamma_{k,\ell}^2$. In fact, this approximation holds for all cells, whether or not they are nonzero.

For these reasons, we define the function $f:\addomain\to\mathbb{R}$ by
\begin{eqnarray*}
	f(X)
	&=&
	\prod_{k=1}^t \left(\sum X_{k,\dummyvar}\right)^{(\sum X_{k,\dummyvar})}
	\prod_{\ell=1}^u \left(\sum X_{\dummyvar,\ell}\right)^{(\sum X_{\dummyvar,\ell})}
	\prod_{k,\ell} \Gamma_{k,\ell}^{2X_{k,\ell}}X_{k,\ell}^{-2X_{k,\ell}},\\
	&=&
	\prod_{k,\ell}
		\left(\frac{\Gamma_{k,\ell}^2 \left(\sum X_{k,\dummyvar}\right)\left(\sum X_{\dummyvar,\ell}\right)}{X_{k,\ell}^2}\right)^{X_{k,\ell}}.
\end{eqnarray*}
For future reference, note that each factor in the product above is always greater than or equal to $1$ since $\sum X_{k,\dummyvar}$ and $\sum X_{\dummyvar,\ell}$ are each greater than or equal to $X_{k, \ell}$ and the growth rate of any infinite permutation class is at least $1$.

The admissible domain $\addomain$ is compact, so we can find a subsequence of $\{A^*_n\}$ for which the sequence $\{A^*_n/n\}$ of scaled matrices converges pointwise to some matrix $\hat{A}^\ast\in\addomain$. Then (as we have assumed that the cells of $\M$ all have proper growth rates) an application of Stirling's Formula followed by taking $n^{\mbox{\scriptsize th}}$ roots and limits shows that the upper growth rate of $\Grid(\M)$ is equal to $f(\hat{A}^\ast)$. (This is in fact the only place in the argument where the hypothesis that the classes $\M_{k,\ell}$ have proper growth rates is used.)

Conversely, for an arbitrary $X\in\addomain$, the same application of Stirling's Formula, $n^{\mbox{\scriptsize th}}$ roots, and limits shows that
\[
	f(X)
	=
	\lim_{n\rightarrow\infty}\ \lvert\Grid^\gridded_{\lfloor nX\rfloor}(\M)\rvert^{1/n}.
\]
We know that $|\Grid^\gridded_{\lfloor nX\rfloor}(\M)|\le |\Grid^\gridded_n(\M)|$ for all $n$, so this computation shows that $f(X)$ is a lower bound on the lower growth rate of $\Grid(\M)$ for all matrices $X$ in the admissible domain. On the other hand, as noted in the preceding paragraph, the upper growth rate of $\Grid(\M)$ is achieved by a value of $f$. These considerations establish the following result, which reduces the problem of computing the growth rate of $\Grid(\M)$ to that of maximizing $f$ on the admissible domain.

\begin{proposition}
\label{prop-gr-maximize-f}
If every cell of $\M$ has a proper growth rate then $\Grid(\M)$ has a proper growth rate, which is equal to the maximum value of the function $f$ on the admissible domain.
\end{proposition}

\section{Singular Values and the Lower Bound}

Proposition~\ref{prop-gr-maximize-f} gives us a straightforward way to establish lower bounds on the growth rate of $\Grid(\M)$: we must simply compute the function $f$ for an admissible matrix $X$. We do so via the singular values of $\Gamma$. We refer the reader to any comprehensive linear algebra book for the full theory of singular value decompositions of matrices and simply state here the facts we require, specialized to our context.

The nonnegative real number $s$ is a \emph{singular value} of the $t\times u$ matrix $\Gamma$ if there exist unit column vectors $\vec{r}\in\mathbb{R}^u$ and $\vec{c}\in\mathbb{R}^t$, called \emph{left} and \emph{right singular vectors}, respectively, such that
\[
	\Gamma^T\vec{r}=s\vec{c}
	\quad\text{and}\quad
	\Gamma\vec{c}=s\vec{r}.
\]
Let $s$ be a singular value of $\Gamma$. From the above it follows that every left singular vector $\vec{r}$ for $s$ is an eigenvector of $\Gamma\Gamma^T$ for the eigenvalue $s^2$, while every right singular vector $\vec{c}$ for $s$ is an eigenvalue of $\Gamma^T\Gamma$, also for the eigenvalue $s^2$. Because $\Gamma$ is nonnegative, the Perron--Frobenius Theorem implies that there are nonnegative left and right singular vectors for the greatest singular value of $\Gamma$.

\begin{proposition}
\label{prop-sing-value-to-lagrange}
Let $\vec{r}$ and $\vec{c}$ be nonnegative left and right singular vectors for the greatest singular value $s$ of $\Gamma$ and define the matrix $X$ by $X_{k,\ell}=\Gamma_{k,\ell}\vec{r}_k\vec{c}_\ell$. Then $X$ lies in the admissible domain and $f(X)=s^2$, the greatest eigenvalue of $\Gamma^T\Gamma$.
\end{proposition}
\begin{proof}
Suppose $\Gamma$ is a $t\times u$ matrix and let $\vec{r}$, $\vec{c}$, $s$, and $X$ be as in the statement of the proposition. It follows that the sum over the $k^{\mbox{\scriptsize th}}$ column of $X$ is
\[
	\sum X_{k,\dummyvar}
	=
	\vec{r}_k\sum_{j=1}^u \Gamma_{k,j}\vec{c}_j.
\]
The sum on the right hand side is the inner product of the $k^{\mbox{\scriptsize th}}$ column of $\Gamma$ with $\vec{c}$, and thus because $\vec{r}$ and $\vec{c}$ are singular vectors of $\Gamma$ corresponding to $s$, the quantity above is equal to $s\vec{r}_k^2$. Analogously, the row sums simplify as $\sum A_{\dummyvar,\ell}=s\vec{c}_\ell^2$.

This establishes that for $(k,\ell)\in\supp(X)$ we have
\[
	\frac{\Gamma_{k,\ell}^2\left(\sum X_{k,\dummyvar}\right)\left(\sum X_{\dummyvar,\ell}\right)}{X_{k,\ell}^2}
	=
	\left(\frac{s \Gamma_{k,\ell} \vec{r}_k\vec{c}_\ell}{\Gamma_{k,\ell}\vec{r}_k\vec{c}_\ell}\right)^2
	=
	s^2,
\]
as desired.
\end{proof}

The \emph{Hadamard product} of the matrices $A$ and $B$ (of the same size) is defined by $(A\circ B)_{k,\ell}=A_{k,\ell}B_{k,\ell}$ and the \emph{tensor product} of the vectors $\vec{r}$ and $\vec{c}$ is defined by $(\vec{r}\otimes\vec{c})_{k,\ell}=\vec{r}_k\vec{c}_{\ell}$. Note that with this notation, the matrix $X$ in Proposition~\ref{prop-sing-value-to-lagrange}, which according to the results of this section and the next gives a blueprint for constructing typical members of $\Grid(\M)$, can be expressed as
\[
	\Gamma\circ(\vec{r}\otimes\vec{c}).
\]

\section{Lagrange Multipliers and the Upper Bound}

Having established the lower bound in Theorem~\ref{thm-bevan-grid-gr}, we seek to find a matching upper bound. This task is equivalent (by Proposition~\ref{prop-gr-maximize-f}) to maximizing $f$, or equivalently, the function $\log f$, on the admissible domain. By slight abuse of notation, for this purpose we view the admissible domain as
\[
	\addomain=\left\{X\in[0,1]^{\supp(\Gamma)} \st \sum_{(k,\ell)\in\supp(\Gamma)} X_{k,\ell}=1\right\}.
\]
As this is a constrained optimization problem, it is natural to use Lagrange multipliers, and indeed this is the approach we take in proving Proposition~\ref{prop-lagrange-to-sing-value}, below, which completes our proof of Theorem~\ref{thm-bevan-grid-gr}.

In preparation for this approach, we compute that, viewing the contents of the cells of a matrix $X\in\addomain$ as formal variables and for a cell $(k,\ell)\in\supp(\Gamma)$,
\begin{equation}
\label{eqn-partial-derivs}\tag{$\dagger$}
	\frac{\partial \log f}{\partial X_{k,\ell}}
	=
	2 \log \Gamma_{k,\ell}
	+
	\underbrace{\left( \log \left(\sum X_{k,\dummyvar}\right) - \log X_{k,\ell} \right)}_{\text{column contribution}}
	+
	\underbrace{\left( \log \left(\sum X_{\dummyvar,\ell}\right) - \log X_{k,\ell} \right)}_{\text{row contribution}}.
\end{equation}
Because the nonempty entries of $\M$ are infinite classes, we have $\Gamma_{k,\ell}\ge 1$ for all $(k,\ell)\in \supp(\Gamma)$. Moreover, each of the column (resp., row) contributions in \eqref{eqn-partial-derivs} is nonnegative, because $\log (\sum X_{k,\dummyvar})$ dominates $\log X_{k,\ell}$ (resp., because $\log (\sum X_{\dummyvar,\ell})$ dominates $\log X_{k,\ell}$). Therefore $\partial\log f/\partial X_{k,\ell}\ge 0$ for all $(k,\ell)\in \supp(\Gamma)$. This should make intuitive sense because increasing $X_{k,\ell}$ corresponds to adding entries to the $(k,\ell)$ cell of the corresponding set of permutations without removing entries from any other cell.

\begin{proposition}
\label{prop-lagrange-to-sing-value}
The maximum value of the function $f$ on the admissible domain $\addomain$ is at most the greatest eigenvalue of $\Gamma^T\Gamma$.
\end{proposition}
\begin{proof}
Suppose that $\hat{A}$ maximizes $f$ on the admissible domain and set $S=\supp(\hat{A})$. There is a degenerate case that must first be ruled out. If $|S|=1$, meaning that $\hat{A}$ has a single nonzero entry (which is therefore equal to $1$), then it can be seen from either the combinatorial or analytic definition of $f$ that $f(\hat{A})$ is equal to $\Gamma_{k,\ell}^2$ for some cell $(k,\ell)$, which is a lower bound on the largest eigenvalue of $\Gamma^T\Gamma$.

With this case dispatched, we may assume that the nonzero entries of $\hat{A}$ are all strictly between $0$ and $1$. The usefulness of this fact is that it shows that $\hat{A}$ lies in the interior of the subset $\{X\in\addomain\st \supp(X)\subseteq S\}$, which we view as
\[
	\addomain_S
	=
	\left\{ X\in[0,1]^S \st \sum_{(k,\ell)\in S} X_{k,\ell}=1\right \}.
\]
For the purposes of this proof we also restrict the sums $\sum \hat{A}_{k,\dummyvar}$ and $\sum \hat{A}_{\dummyvar,\ell}$ to range over all $\ell$ (resp., $k$) such that $(k,\ell)\in S$.

As $\hat{A}$ maximizes $f$ over $\addomain$, it certainly must maximize $f$ over $\addomain_S$, and since $\hat{A}$ lies in the interior of $\addomain_S$, we know that the Lagrange conditions on $\log f$ must be satisfied at $\hat{A}$. Because the constraint on $\addomain_S$ is simply that the entries in $S$ sum to $1$, these conditions are that the quantities $\partial \log f/\partial X_{k,\ell}$ evaluated at $\hat{A}$ are equal for all $(k,\ell)\in S$. By exponentiating both sides and then taking their square roots, this is equivalent to the statement that there is a value $s>0$ such that for all $(k,\ell)\in S$,
\begin{equation}
\label{eqn-weights-goal}\tag{$\ddagger$}
	\frac{\Gamma_{k,\ell}\, \sqrt{\sum \hat{A}_{k,\dummyvar}}\, \sqrt{\sum \hat{A}_{\dummyvar,\ell}}}{\hat{A}_{k,\ell}}
	=
	s.
\end{equation}
Indeed, if \eqref{eqn-weights-goal} holds, then we see that $f(\hat{A})=s^2$. Therefore our last step is to show that $s^2$ is at most the greatest eigenvalue of $\Gamma^T\Gamma$, or equivalently, that $s$ is at most the greatest singular value of $\Gamma$. To this end, define the vectors
\[
	\vec{r}_k=\sqrt{\sum \hat{A}_{k,\dummyvar}}
	\quad\text{and}\quad
	\vec{c}_\ell=\sqrt{\sum \hat{A}_{\dummyvar,\ell}}.
\]
Both $\vec{r}$ and $\vec{c}$ are unit vectors because their norms are equal to the square root of the weight of $\hat{A}$. Furthermore, solving \eqref{eqn-weights-goal} for $\hat{A}_{k,\ell}$ shows that $\hat{A}_{k,\ell}=\Gamma_{k,\ell}\vec{r}_k\vec{c}_\ell/s$. Thus when $\Gamma_{k,j}\neq 0$ we have
\[
	\vec{c}_j = \frac{s \hat{A}_{k,j}}{\Gamma_{k,j}\vec{r}_k}.
\]
This allows us to check that $\Gamma\vec{c}=s\vec{r}$ and $\Gamma^T\vec{r}=s\vec{c}$. For instance,
\[
	(\Gamma\vec{c})_k
	=
	\sum_{j=1}^u \Gamma_{k,j}\vec{c}_j
	=
	\frac{s}{\vec{r}_k}\sum \hat{A}_{k,\dummyvar},
\]
and as the row sum $\sum \hat{A}_{k,\dummyvar}$ is equal to $\vec{r}_k^2$ we have that $(\Gamma\vec{c})_k=s\vec{r}_k$. A similar computation shows that $\Gamma^T\vec{r}=s\vec{c}$.

This establishes that $s$ is a singular value of $\Gamma$ and hence $s^2$ is an eigenvalue of $\Gamma^T\Gamma$. Since $f(\hat{A})=s^2$, this completes the proof of the theorem.
\end{proof}

%
%
%
%
%
%

%

\bibliographystyle{acm}
\bibliography{../../refs}

\end{document}